\documentclass[12pt,reqno]{amsart}

\usepackage{verbatim}
\usepackage{amssymb}
\usepackage{ifthen}

\newcommand{\seq}{\subseteq}
\newcommand{\stm}{\setminus}
\newcommand{\est}{\varnothing}

\newcommand{\alp}{\alpha}

\newcommand{\eps}{\varepsilon}
\newcommand{\kap}{\kappa}
\newcommand{\lam}{\lambda}
\renewcommand{\phi}{\varphi}
\newcommand{\sig}{\sigma}

\newcommand{\Gam}{{\Gamma}}
\newcommand{\Del}{{\Delta}}
\newcommand{\Ome}{{\Omega}}

\newcommand{\C}{{\mathbb C}}
\newcommand{\R}{{\mathbb R}}

\newcommand{\oA}{{\overline A}}

\DeclareMathOperator{\rk}{rk}

\renewcommand{\(}{\left(}
\renewcommand{\)}{\right)}

\newcommand{\<}{\langle}
\renewcommand{\>}{\rangle}

\newcommand{\longc}{,\ldots,}
\newcommand{\longp}{+\dotsb+}

\newcommand{\sub}[1]{_{\substack{#1}}}

\newcommand{\refl}[1]{\ref{l:#1}}
\newcommand{\refc}[1]{\ref{c:#1}}
\newcommand{\reft}[1]{\ref{t:#1}}
\newcommand{\refs}[1]{\ref{s:#1}}
\newcommand{\refb}[1]{\cite{b:#1}}
\newcommand{\refe}[1]{\eqref{e:#1}}

\newtheorem{lemma}{Lemma}
\newtheorem{theorem}{Theorem}
\newtheorem{corollary}{Corollary}

\newcommand{\showall}{no}          

\title[Discrete Norms]%
  {Discrete Norms of a Matrix \\ and the Converse to the Expander Mixing Lemma}
\author{Vsevolod F. Lev}
\address{Department of Mathematics, The university of Haifa at Oranim,
  Tivon 36006, Israel}
\email{seva@math.haifa.ac.il}

\subjclass[2010]%
  {Primary: 05C50;   
   Secondary: 15A18, 
              15A60} 

\begin{document}
\baselineskip = 16pt

\begin{abstract}
We define the \emph{discrete norm} of a complex $m\times n$ matrix $A$ by
  $$ \|A\|_\Del := \max_{0\ne\xi\in\{0,1\}^n} \frac{\|A\xi\|}{\|\xi\|}, $$
and show that
  $$ \frac c{\sqrt{\log h(A)+1}}\,\|A\| \le \|A\|_\Del \le \|A\|, $$
where $c>0$ is an explicitly indicated absolute constant,
$h(A)=\sqrt{\|A\|_1\|A\|_\infty}/\|A\|$, and $\|A\|_1,\|A\|_\infty$, and
$\|A\|=\|A\|_2$ are the induced operator norms of $A$. Similarly, for the
\emph{discrete Rayleigh norm}
  $$ \|A\|_P := \max_{\sub{0\ne\xi\in\{0,1\}^m \\ 0\ne\eta\in\{0,1\}^n}}
                                     \frac{|\xi^tA\eta|}{\|\xi\|\|\eta\|} $$
we prove the estimate
  $$ \frac c{\log h(A)+1}\,\|A\| \le \|A\|_P \le \|A\|. $$
These estimates are shown to be essentially best possible.

As a consequence, we obtain another proof of the (slightly sharpened and
generalized version of the) converse to the expander mixing lemma by
Bollob\'as-Nikiforov and Bilu-Linial.
\end{abstract}

\maketitle

\section{Summary of results}\label{s:intro}

For a complex matrix $A$ with $n$ columns, we define the \emph{discrete norm}
of $A$ by
  $$ \|A\|_\Del := \max_{0\ne\xi\in\{0,1\}^n} \frac{\|A\xi\|}{\|\xi\|}, $$
where the maximum is over all non-zero $n$-dimensional binary vectors $\xi$,
and $\|\cdot\|$ denotes the usual Euclidean vector norm. Recalling the standard
definition of the induced operator $L^2$-norm
  $$ \|A\| := \sup_{0\ne x\in\C^n} \frac{\|Ax\|}{\|x\|}, $$
we see at once that $\|A\|_\Del\le\|A\|$, and one can expect that, moreover,
the two norms are not far from each other.

\subsection{Norm estimates}
Our first goal is to establish a result along the lines just indicated; to
state it, we introduce the notion of a \emph{height} of a matrix.

For $p\in[1,\infty]$, let $\|A\|_p$ denote the induced operator $L^p$-norm of
the matrix $A$:
  $$ \|A\|_p := \sup_{0\ne x\in\C^n} \frac{\|Ax\|_p}{\|x\|_p}, $$
where $n$ is the number of columns of $A$. We are actually interested in the
following three special cases: the \emph{column norm} $\|A\|_1$, which can be
equivalently defined as the largest absolute column sum of $A$; the \emph{row
norm} $\|A\|_\infty$, which is the largest absolute row sum of $A$; and the
Euclidean norm $\|A\|_2$, commonly denoted simply by $\|A\|$. These three
norms are known to be related by the inequality
\begin{equation}\label{e:RieszThorin}
  \|A\|^2 \le \|A\|_1\|A\|_\infty,
\end{equation}
which can be obtained as a particular case of the Riesz-Thorin theorem, or
proved directly, using basic properties of matrix norms (in particular,
sub-multiplicativity of the $L^1$-norm):
 $$ \|A\|^2 = \|A^*A\| \le \|A^*A\|_1
                           \le \|A^*\|_1 \|A\|_1 = \|A\|_\infty \|A\|_1. $$
Also, if $A$ has $m$ rows and $n$ columns, then
\begin{equation}\label{e:sqrt-mn}
  \|A\|_1 \le \sqrt m\,\|A\|\ \text{and}\ \|A\|_\infty \le \sqrt n\,\|A\|.
\end{equation}
We now define the \emph{height} of a non-zero complex matrix $\|A\|$ by
  $$ h(A) := \sqrt{\|A\|_1\|A\|_\infty}/\|A\|; $$
thus, if $A$ is of size $m\times n$, then in view of~\refe{RieszThorin}
and~\refe{sqrt-mn},
\begin{equation}\label{e:hlemn}
  1 \le h(A) \le \sqrt[4]{mn}.
\end{equation}

Having defined the heights, we can state our principal results.
\begin{theorem}\label{t:Dnorm}
For any non-zero complex matrix $A$, we have
  $$  \frac{\|A\|}{8\sqrt{2}\sqrt{\log h(A)+2}} \le \|A\|_\Del \le \|A\|. $$
\end{theorem}

In a similar vein, we define the \emph{discrete Rayleigh norm} of a complex
$m\times n$ matrix $A$ by
  $$ \|A\|_P := \max_{\sub{0\ne\xi\in\{0,1\}^m \\ 0\ne\eta\in\{0,1\}^n}}
                                     \frac{|\xi^tA\eta|}{\|\xi\|\|\eta\|} $$
(where the subscript $P$ stands for the capital Greek letter \emph{rho}), and
prove
\begin{theorem}\label{t:Pnorm}
For any non-zero complex matrix $A$, we have
  $$  \frac{\|A\|}{32\sqrt{2}\,(\log h(A)+4)} \le \|A\|_P \le \|A\|. $$
\end{theorem}
We remark that the trivial upper bounds in Theorems~\reft{Dnorm}
and~\reft{Pnorm} are included solely for comparison purposes. The proofs of
the theorems are presented in Section~\refs{DPnorms}.

Theorem~\reft{Dnorm} to our knowledge has never appeared in the literature,
while Theorem~\reft{Pnorm} extends and refines results of Bollob\'as and
Nikiforov \refb{bn}, and Bilu and Linial \refb{bili}. Specifically, somewhat
hidden in the proof of \cite[Theorem~2]{b:bn} is the assertion that if $A$ is
Hermitian of order $n\ge 2$, then $\|A\|_P\gg\|A\|/\log n$, and
\cite[Lemma~3.3]{b:bili} essentially says that if $A$ is a symmetric real
matrix with the diagonal entries sufficiently small in absolute value, then
$\|A\|_P\gg\|A\|/\big(\log(\|A\|_\infty/\|A\|_P)+1)$. (The notation $X\ll Y$
will be used throughout to indicate that there is an absolute constant $C$
such that $|X|\le C|Y|$.) The former of these results follows from
Theorem~\reft{Pnorm} in view of \refe{hlemn}; to derive the latter just
observe that for $A$ symmetric,
  $$ h(A) = \|A\|_\infty/\|A\| \le \|A\|_\infty/\|A\|_P. $$
It is worth pointing out that our argument is completely distinct from those
used in \refb{bn} and \refb{bili}.

As an application, consider the situation where $A$ is the adjacency matrix
of an undirected graph; thus, $\|A\|$ is the spectral radius of the graph,
and $\|A\|_1=\|A\|_\infty$ is its maximum degree. Identifying the vectors
$\xi,\eta\in\{0,1\}^n$ in the definitions of the discrete norms with the
corresponding subsets of the vertex set of the graph, as an immediate
consequence of Theorems \reft{Dnorm} and \reft{Pnorm} we get the following
corollaries allowing one to interpret the spectral radius combinatorially.
\begin{corollary}\label{c:Dnorm}
Let $(V,E)$ be a graph with the spectral radius $\rho$ and maximum degree
$\Del$. For a vertex $v\in V$ and a subset $X\seq V$, denote by $N_X(v)$ the
set of all neighbors of $v$ in $X$:
  $$ N_X(v) := \{ u\in V\colon uv\in E \}. $$
Then for any subset $X\seq V$ we have
  $$ \sum_{v\in V} |N_X(v)|^2 \le \rho^2 |X|, $$
and there exists a non-empty subset $X\seq V$ such that
  $$ \sum_{v\in V} |N_X(v)|^2
              \ge \frac{\rho^2}{128\big(\log(\Del/\rho)\,+2\big)}\, |X|. $$
\end{corollary}

\begin{corollary}\label{c:Pnorm}
Let $(V,E)$ be a graph with the spectral radius $\rho$ and maximum degree
$\Del$. For subsets $X,Y\seq V$, denote by $e(X,Y)$ the number of edges
joining a vertex from $X$ with a vertex from $Y$, those edges having both
their endpoints in $X\cap Y$ being counted twice:
  $$ e(X,Y) := | \{ (x,y)\in X\times Y\colon xy\in E \} |. $$
Then for any subsets $X,Y\seq V$ we have
  $$ e(X,Y) \le \rho \sqrt{|X||Y|}, $$
and there exist non-empty subsets $X,Y\seq V$ such that
  $$ e(X,Y) \ge
       \frac{\rho}{32\sqrt2\big(\log(\Del/\rho)\,+4\big)}\, \sqrt{|X||Y|}. $$
\end{corollary}

\subsection{Second singular value estimates}
For a complex matrix $A$, let $\sig_2(A)$ denote its second singular value;
thus, for instance, if $A$ is Hermitian of order $n$ with the eigenvalues
$\lam_1\longc\lam_n$, then $\sig_2(A)$ is the second largest among the
absolute values $|\lam_1|\longc|\lam_n|$. By the second singular value of a
\emph{graph} we will mean the second singular value of its adjacency matrix.

From the singular value decomposition theorem it is easy to derive that if
$D$ is a matrix of the same size as $A$ and rank at most $1$, then
\begin{equation}\label{e:EYM}
  \|A-D\|\ge\sig_2(A);
\end{equation}
this is a particular case of the Eckart-Young-Mirsky theorem \refb{m} (see
also \refb{s} for the history of this theorem which has been re-discovered a
number of times). Below we choose $D$ to be the matrix all of whose elements
are equal to the arithmetic mean of the elements of $A$; we denote this
matrix by $\oA$. It is readily verified that $\|\oA\|_1\le\|A\|_1$ and
$\|\oA\|_\infty\le\|A\|_\infty$, whence, in view of \refe{EYM} and assuming
$\rk A\ge 2$,
\begin{equation}\label{e:hA-D}
  h(A-\oA) = \sqrt{\|A-\oA\|_1\|A-\oA\|_\infty} / \|A-\oA\|
                                 \le 2\sqrt{\|A\|_1\|A\|_\infty} / \sig_2(A).
\end{equation}
On the other hand, from \refe{EYM} and Theorem~\reft{Dnorm} we get
\begin{equation}\label{e:sig2leA-D}
  \sig_2(A) \le \|A-\oA\| \le 8\sqrt2\sqrt{\log h(A-\oA)+2}
                                                       \,\cdot\|A-\oA\|_\Del.
\end{equation}
Combining \refe{hA-D} and \refe{sig2leA-D}, we obtain
\begin{theorem}\label{t:sig2D}
Suppose that $A$ is a complex matrix of rank at least $2$, and let $\oA$ be
the identically-sized matrix all of whose elements are equal to the
arithmetic mean of the elements of $A$. Then, writing
$K:=2\sqrt{\|A\|_1\|A\|_\infty}/\sig_2(A)$, we have
  $$ \|A-\oA\|_\Del \ge \frac{\sig_2(A)}{8\sqrt 2\sqrt{\log K+2}}. $$
\end{theorem}

Arguing the same way but using Theorem~2 instead of Theorem~1, we get
\begin{theorem}\label{t:sig2P}
Suppose that $A$ is a complex matrix of rank at least $2$, and let $\oA$ be
the identically-sized matrix all of whose elements are equal to the
arithmetic mean of the elements of $A$. Then, writing
$K:=2\sqrt{\|A\|_1\|A\|_\infty}/\sig_2(A)$, we have
  $$ \|A-\oA\|_P \ge \frac{\sig_2(A)}{32\sqrt 2(\log K+4)}. $$
\end{theorem}

Specifying Theorems~\reft{sig2D} and~\reft{sig2P} to the case where $A$ is
the adjacency matrix of a graph, we obtain the following corollaries (stated
in terms of the second singular value of a graph which, we recall, is the
second largest among the absolute values of its eigenvalues).
\begin{corollary}\label{c:sig2D}
Let $(V,E)$ be a non-empty graph with the maximum degree $\Del$, average
degree $d$, and the second singular value $\sig$. Then there exists a
non-empty subset $X\seq V$ such that, with $N_X(v)$ as in Corollary
\refc{Dnorm}, we have
  $$ \sum_{v\in V} \left(|N_X(v)|-d\,\frac{|X|}{|V|}\right)^2
                         \ge \frac{\sig^2}{128(\log(2\Del/\sig)+2)}\,|X|. $$
\end{corollary}

Corollary~\refc{sig2D} is a converse to a result of Alon and Spencer
\cite[Theorem~9.2.4]{b:as} asserting that, under the notation of the
corollary, if $(V,E)$ is $d$-regular, then
  $$ \sum_{v\in V} \left(|N_X(v)|-d\,\frac{|X|}{|V|}\right)^2
                            \le \left(1-\frac{|X|}{|V|}\right)\sig^2\,|X| $$
for any $X\seq V$.

It is not difficult to see that the second singular value of a non-empty
graph is at least $1$; thus, the ratio $2\Del/\sig$ in the statement of
Corollary~\refc{sig2D} (and also Corollary~\refc{sig2P} immediately
following) does not exceed $2\Del$.

\begin{corollary}\label{c:sig2P}
Let $(V,E)$ be a non-empty graph with the maximum degree $\Del$, average
degree $d$, and the second singular value $\sig$. Then there exist non-empty
subsets $X,Y\seq V$ such that, with $e(X,Y)$ as in Corollary~\refc{Pnorm}, we
have
  $$ \left| e(X,Y) - d\,\frac{|X||Y|}{|V|} \right| \ge
       \frac{\sig}{32\sqrt2\big(\log(2\Del/\sig)\,+4\big)}\,\sqrt{|X||Y|}. $$
\end{corollary}

Corollary \refc{sig2P} is a converse to the well-known Expander Mixing Lemma
(see, for instance, \cite[Corollary~9.2.5]{b:as}) which says that if $(V,E)$
is $d$-regular, then
  $$ \left| e(X,Y) - d\,\frac{|X||Y|}{|V|} \right| \le \sig \sqrt{|X||Y|} $$
for all $X,Y\seq V$.

We remark that Theorem~\reft{sig2P} and Corollary~\refc{sig2P} are rather
close to \cite[Theorem~2]{b:bn} and \cite[Corollary~5.1]{b:bili},
respectively. Namely,  \cite[Theorem~2]{b:bn} says that, in our notation, if
$A$ is Hermitian of order $n\ge 2$, then
\begin{equation}\label{e:bnsig2}
  \|A-\oA\|_P\gg\sig_2(A)/\log n,
\end{equation}
while \cite[Corollary~5.1]{b:bili} essentially says that if $(V,E)$ is a
$d$-regular graph with the second singular value $\sig$ satisfying
\begin{equation}\label{e:bilisig2-a}
  \left| e(X,Y) - d\,\frac{|X||Y|}{|V|} \right| \le \alp \sqrt{|X||Y|}
\end{equation}
for all $X,Y\seq V$, with some $0<\alp\le d$, then
\begin{equation}\label{e:bilisig2}
  \alp \gg \sig / \big(\log(d/\alp)\,+1\big).
\end{equation}
It is readily seen that Corollary~\refc{sig2P} implies \refe{bilisig2}: for
if $\alp\le\sig$, then $\log(2d/\sig)\le\log(d/\alp)+1$, whence from
Corollary~\refc{sig2P} and \refe{bilisig2-a},
  $$ \alp \ge \frac{\sig}{32\sqrt2\big(\log(2d/\sig)\,+4\big)}
                                     \gg \frac{\sig}{\log(d/\alp)+1}. $$
As to \refe{bnsig2}, it cannot be formally derived from Theorem~\reft{sig2P},
but follows easily from Theorem~\reft{Pnorm} and the estimates \refe{hlemn}
and \refe{EYM}:
  $$ \|A-\oA\|_P \gg \frac{\|A-\oA\|}{\log h(A-\oA)+1}
                                            \ge \frac{\sig_2(A)}{\log n}. $$

\bigskip
\subsection{Sharpness}
Theorems~\reft{Dnorm} and~\reft{Pnorm} are sharp in the sense that the
logarithmic function in their lower bounds cannot be replaced with any slower
growing function. To see this, for integer $n\ge 4$ consider the vector
$x=(1,1/\sqrt2\longc1/\sqrt n)^t$, and let $A=xx^t$; thus, $A$ is a symmetric
real matrix of order $n$ with the entries $1/\sqrt{ij}\ (i,j\in[1,n])$. It is
readily verified that $\|A\|_1=\|A\|_\infty<2\sqrt n$ and $Az=\<x,z\>x$,
whence $\|A\|=\|x\|^2>\log n$ and therefore $h(A)<2\sqrt n/\log n$.
Consequently, for every non-zero vector $\xi\in\{0,1\}^n$, writing
$k:=\|\xi\|^2$, we have
\begin{multline*}
  \|A\xi\| =   \<x,\xi\> \|x\|
           \le \left( 1+\frac1{\sqrt 2} \longp \frac1{\sqrt k} \right)
                                                      \frac1{\|x\|}\|A\| \\
           <   \frac{2}{\sqrt{\log n}}\,\|A\|\sqrt{k}
           < \frac{2}{\sqrt{\log h(A)}}\,\|A\|\|\xi\|,
\end{multline*}
implying
  $$ \|A\|_\Del < \frac2{\sqrt{\log h(A)}}\,\|A\|. $$
Similarly, for all non-zero $\xi,\eta\in\{0,1\}^n$, writing $k:=\|\xi\|^2$
and $l:=\|\eta\|^2$, we have
  $$ |\xi^tA\eta|=\<\xi,x\>\<\eta,x\> < 2\sqrt{k}\cdot 2\sqrt{l}
               < \frac4{\log n} \|A\|\|\xi\|\|\eta\|
                                  < \frac4{\log h(A)} \|A\|\|\xi\|\|\eta\| $$
whence
  $$ \|A\|_P < \frac4{\log h(A)}\,\|A\|. $$

Furthermore, Bollob\'as and Nikiforov \cite[Section~3]{b:bn} construct
regular graphs $(V,E)$ of arbitrarily large even order $n:=|V|$ and degree
$n/2$ such that, denoting by $A$ the adjacency matrix of $(V,E)$, and by
$\oA$ the square matrix of order $n$ with all elements equal to $1/2$ (which
is the average of the elements of $A$), one has $\|A-\oA\|_P\ll\sig_2(A)/\log
n$; this shows that the logarithmic factors in Theorem~\reft{sig2P} and
Corollary~\refc{sig2P} cannot be replaced with sub-logarithmic ones. Another
example of this sort is given by Bilu and Linial \cite[Theorem~5.1]{b:bili}.
Although we have not checked carefully the details, we believe that the
constructions of Bollob\'as-Nikiforov and Bilu-Linial can also be used to
show that Theorem~\reft{sig2D} and Corollary~\refc{sig2D} are tight.

An interesting question not addressed by these observations is whether
Corollaries~\refc{Dnorm} and~\refc{Pnorm} are sharp; that is, whether one can
improve Theorems~\reft{Dnorm} and~\reft{Pnorm} under the extra assumption
that the matrix $A$ under consideration is zero-one and symmetric. Notice
that if $A$ corresponds to a \emph{regular} graph, then the norm $\|A\|$ is
equal to the degree of the graph, and taking the vectors $\xi$ and $\eta$ in
the definitions of discrete norms to be the all-$1$ vectors, we see that in
this case $\|A\|=\|A\|_\Del=\|A\|_P$. Consequently, any example showing that
the logarithmic factors in Corollaries~\refc{Dnorm} and~\refc{Pnorm} cannot
be dropped should involve highly non-regular graphs. In this direction we
prove the following result, giving at least a partial solution to the
problem.

\begin{theorem}\label{t:Kn}
For integer $m\ge 1$, let $\Gam_m$ be the graph on the set $\{0,1\}^m$ of all
binary vectors of length $m$, with two vectors adjacent if and only if they
have disjoint supports. Then, denoting by $A_m$ the adjacency matrix of
$\Gam_m$, we have $\|A_m\|_\Del\ll\|A_m\|/\sqrt[4]{m}$ and
$\|A_m\|_P\ll\|A_m\|/\sqrt{m}$, with absolute implicit constants.
\end{theorem}

The graph of Theorem~\reft{Kn} is similar to the well-known Kneser graphs;
however, unlike the ``standard'' Kneser graphs, the vertex set of our graph
is not restricted to vectors of fixed weight. The graph is simple, except for
the loop attached to the zero vector; clearly, removing this loop will not
affect significantly any of the norms in question.

\bigskip
We now turn to the proofs. Theorems~\reft{Dnorm} and~\reft{Pnorm} are proved
in the next section; as we have explained above, Theorems~\reft{sig2D}
and~\reft{sig2P}, as well as Corollaries~\refc{Dnorm}--\refc{sig2P}, are
their direct consequences, and will not be addressed any more.
Theorem~\reft{Kn} is proved in Section~\refs{Kn}.

\section{Proofs of Theorems~\reft{Dnorm} and~\reft{Pnorm}}\label{s:DPnorms}

Both proofs share the same toolbox: Lemma~\refl{tt} showing that for any
complex matrix $A$, there exists a vector $z$ with $\|Az\|/\|z\|$ close to
$\|A\|$ and the ratios of its non-zero coordinates bounded in terms of the
height $h(A)$, and Lemmas~\refl{ell-nonneg}--\refl{ell-complex} showing that
a low-height vector cannot be approximately orthogonal to all binary vectors
simultaneously.

For a non-zero vector $z=(z_1\longc z_n)\in\C^n$, we define the
\emph{logarithmic diameter} of $z$ by
  $$ \ell(z) := \frac{\max\{|z_i|\colon i\in[n]\}}%
                                 {\min\{|z_i|\colon i\in[n],\ z_i\ne 0\}}. $$

\begin{lemma}\label{l:tt}
Let $n\ge 1$ be an integer and $K\ge 1$ a real number. For any non-zero
complex matrix $A$ with $n$ columns of height $h(A)\le K$, there exists a
vector $z\in\C^n$ such that $\|Az\|>\frac12\|A\|\|z\|$ and $\ell(z)<8K^2+1$.
\end{lemma}

\begin{proof}
Fix a unit-length vector $x=(x_1\longc x_n)^t\in\C^n$ with $\|Ax\|=\|A\|$ and
let $M:=8K^2+1$. Consider the decomposition
  $$ x = \sum_{k=-\infty}^\infty x^{(k)}, $$
where for every integer $k$, the vector $x^{(k)}=(x^{(k)}_1\longc
x^{(k)}_n)^t$ is defined by
  $$ x^{(k)}_i := \begin{cases}
       x_i &\text{ if}\ M^k\le |x_i|<M^{k+1}, \\
       0   &\text{ otherwise}.
                  \end{cases} $$
Notice, that $\ell(x^{(k)})<M$ whenever $x^{(k)}\ne 0$. We have
\begin{equation}\label{e:Axx-large}
  \|A\|^2 = |\<Ax,Ax\>| \le \sum_{k,l=-\infty}^\infty |\<Ax^{(k)},Ax^{(l)}\>|
\end{equation}
and, since the vectors $x^{(k)}$ are pairwise orthogonal,
  $$ \sum_{k=-\infty}^\infty \|x^{(k)}\|^2 = \|x\|^2 = 1. $$
Since $h(A)\le K$ implies
  $$ |\<Au,Av\>| \le \|Au\|_\infty\|Av\|_1
             \le (\|A\|_\infty\|A\|_1) \|u\|_\infty\|v\|_1
                                         \le K^2\|A\|^2\|u\|_\infty\|v\|_1 $$
for all $u,v\in\C^n$, the contribution to the right-hand side of
\refe{Axx-large} of the summands with $l\ge k+2$ can be estimated as follows:
\begin{align*}
  \sum_{k,l\colon l\ge k+2} |\<Ax^{(k)},Ax^{(l)}\>|
    &\le K^2\|A\|^2 \sum_{k,l\colon l\ge k+2} M^{k+1}
                           \sum_{i\in[n]\colon M^l\le |x_i|<M^{l+1}} |x_i| \\
    &= K^2\|A\|^2 \sum_{l=-\infty}^\infty \sum_{k=-\infty}^{l-2} M^{k+1}
                           \sum_{i\in[n]\colon M^l\le |x_i|<M^{l+1}} |x_i| \\
    &\le \frac{K^2}{M-1}\,\|A\|^2\, \sum_{l=-\infty}^\infty\ %
                         \sum_{i\in[n]\colon M^l\le |x_i|<M^{l+1}} |x_i|^2 \\
    &=   \frac18\,\|A\|^2\|x\|^2 \\
    &=   \frac18\,\|A\|^2.
\end{align*}
By symmetry,
\begin{equation}\label{e:off-diag}
  \sum_{k,l\colon|k-l|\ge 2} |\<Ax^{(k)},Ax^{(l)}\>| \le \frac14\,\|A\|^2.
\end{equation}

Assuming that the assertion of the lemma fails to hold, we have
$\|Ax^{(k)}\|\le\frac12\,\|A\|\|x^{(k)}\|$ for every integer $k$. Hence,
under this assumption, for any fixed integer $d$,
\begin{align*}
  \sum_{k,l\colon k-l=d} |\<Ax^{(k)},Ax^{(l)}\>|
    &\le \frac14\,\|A\|^2 \sum_{l=-\infty}^\infty \|x^{(l)}\|\|x^{(l+d)}\|\\
    &\le   \frac14\,\|A\|^2\sum_{l=-\infty}^\infty \|x^{(l)}\|^2 \\
    &= \frac14\,\|A\|^2,
\end{align*}
the second inequality being strict unless $d=0$. It follows that
  $$ \sum_{k,l\colon|k-l|\le 1} |\<Ax^{(k)},Ax^{(l)}\>| < \frac34\,\|A\|^2; $$
along with \refe{off-diag} this yields
  $$ \sum_{k,l=-\infty}^\infty |\<Ax^{(k)},Ax^{(l)}\>| < \|A\|^2, $$
contradicting \refe{Axx-large}.
\end{proof}

For non-zero vectors $u,v\in\C^n$, we write $\cos(u,v):=\<u,v\>/\|u\|\|v\|$.
\begin{lemma}\label{l:ell-nonneg}
Let $n\ge 1$ be an integer and $K\ge 1$ a real number. If $z\in\R^n$ is a
vector with non-negative coordinates and logarithmic diameter $\ell(z)\le K$,
then there exists a binary vector $\xi\in\{0,1\}^n$ such that $\cos(z,\xi)\ge
1/\sqrt{\log K+1}$.
\end{lemma}

\begin{proof}
Passing to the appropriate coordinate subspace and scaling the vector $z$, we
assume that all its coordinates are between $1$ and $K$. For $t\ge 0$, denote
by $\Phi(t)$ the number of those coordinates which are greater than or equal
to $t$, and let $\xi_t\in\{0,1\}^n$ be the characteristic vector of this set
of coordinates; thus $\|\xi_t\|^2=\Phi(t)$. Also, straightforward
verification shows that
\begin{align}
  \int_t^K \Phi(\tau)\,d\tau &= \<z,\xi_t\> - t\Phi(t),
                                          \quad t\in[0,K] \label{e:int1} \\
\intertext{and}
  \int_1^K 2\tau\Phi(\tau)\,d\tau &= \|z\|^2-n. \label{e:int2}
\end{align}
Let $\kap:=1/\sqrt{\log K+1}$. From \refe{int1} we get
 $\<z,\xi_t\>\ge t\Phi(t)=t\|\xi_t\|^2$; consequently, if the
assertion of the lemma were wrong, for each $t>0$ we would have
  $$ \<z,\xi_t\>^2 \le \kap^2\|z\|^2\|\xi_t\|^2
                              \le \kap^2\|z\|^2\cdot\frac1t\<z,\xi_t\>; $$
hence
  $$ \<z,\xi_t\> \le \frac1t\,\kap^2\|z\|^2,\quad t>0. $$
Substituting this estimate into \refe{int1}, integrating over $t$ in the
range $[1,K]$, using \refe{int1} and \refe{int2}, and taking into account
that $\Phi(1)=n$ and $\<\xi_1,z\>=\|z\|_1$, we obtain
\begin{align}
  \kap^2\|z\|^2\log K
    &\ge \int_1^K \left( \int_t^K\Phi(\tau)\,d\tau\right) \,dt
                                        + \int_1^K t\Phi(t) \,dt \notag \\
    &=   \int_1^K (\tau-1)\Phi(\tau)\,d\tau  + \int_1^K t\Phi(t) \,dt \notag \\
    &=   \int_1^K 2\tau\Phi(\tau)\,d\tau - \int_1^K\Phi(\tau)\,d\tau \notag \\
    &=   \|z\|^2-\|z\|_1. \label{e:logK}
\end{align}
From the assumption that the assertion of lemma in wrong we get
  $$ \<z,\xi_1\> < \kap\|z\|\|\xi_1\| $$
(for otherwise the assertion would hold true with $\xi=\xi_1$). As a result,
  $$ \|z\|_1^2 = \<z,\xi_1\>^2 < \kap^2\|z\|^2\|\xi_1\|^2
                               = \kap^2\|z\|^2 n \le \kap^2\|z\|^2\|z\|_1, $$
whence
  $$ \|z\|_1 < \kap^2\|z\|^2. $$
Substituting  into \refe{logK} we get
  $$ \kap^2\|z\|^2\log K > \|z\|^2 -\kap^2\|z\|^2, $$
in a contradiction with our choice of $\kap$.
\end{proof}

Lemma \refl{ell-nonneg} is easy to extend onto \emph{arbitrary} real vectors
(which may have some of their coordinates negative).
\begin{lemma}\label{l:ell-arb}
Let $n\ge 1$ be an integer and $K\ge 1$ a real number. If $z\in\R^n$ is a
vector with the logarithmic diameter $\ell(z)\le K$, then there exists
$\xi\in\{0,1\}^n$ such that
 $|\cos(z,\xi)|\ge 1/\sqrt{2(\log K+1)}$.
\end{lemma}

\begin{proof}
Write $z=z^+-z^-$, where $z^+$ and $z^-$ have non-negative coordinates and
disjoint supports. Observing that $\|z^+\|^2+\|z^-\|^2=\|z\|^2$, choose
$z'\in\{z^+,z^-\}$ with $\|z'\|\ge\|z\|/\sqrt 2$. Clearly, we have
$\ell(z')\le\ell(z)\le K$; therefore, by Lemma~\refl{ell-nonneg}, there
exists $\xi\in\{0,1\}^n$ with
  $$ \cos(z',\xi) \ge \frac1{\sqrt{\log K+1}}. $$
Assuming without loss of generality that for any vanishing coordinate of
$z'$, the corresponding coordinate of $\xi$ also vanishes, we then get
  $$ |\<z,\xi\>| = \<z',\xi\>
           \ge \frac1{\sqrt{\log K+1}}\,\|z'\|\|\xi\|
                           \ge \frac1{\sqrt{2(\log K+1)}}\,\|z\|\|\xi\|, $$
proving the assertion.
\end{proof}

For the remainder of this section, we extend the notion of height of a matrix
(introduced in Section~\refs{intro}) onto vectors by identifying them with
one-column or one-row matrices; that is, the height of a non-zero complex
vector $z$ is
  $$ h(z) := \sqrt{\|z\|_1\|z\|_\infty}/\|z\|. $$
We now prove a version of Lemma~\refl{ell-arb} which applies to a wider class
of vectors; namely, real vectors of bounded height (instead of the bounded
logarithmic diameter).
\begin{lemma}\label{l:real-height}
Let $n\ge 1$ be an integer and $K\ge 1$ a real number. If $z\in\R^n$ is a
vector of height $h(z)\le K$, then there exists $\xi\in\{0,1\}^n$ such that
$|\cos(z,\xi)|\ge 1/(2\sqrt{\log(2K^2)+1})$.
\end{lemma}

\begin{proof}
Let $M:=\|z\|^2/\|z\|_1$. Writing $z=(z_1\longc z_n)^t$, we have
  $$ \sum_{i\colon |z_i|<M/2} z_i^2 \le \frac12\,M\|z\|_1
                                                       = \frac12\/\|z\|^2, $$
whence
\begin{equation}\label{e:half-norm}
  \sum_{i\colon |z_i|\ge M/2} z_i^2 \ge \frac12\, \|z\|^2.
\end{equation}
Consider the vector $z'=(z_1'\longc z_n')^t$ defined by
  $$ z_i' = \begin{cases}
               z_i &\text{if}\ |z_i| \ge M/2, \\
               0   &\text{if}\ |z_i| < M/2,
             \end{cases} $$
for each $i\in[n]$. Since $\ell(z')\le\|z\|_\infty/(M/2)=2h^2(z)\le 2K^2$, by
Lemma~\refl{ell-arb} there exists $\xi\in\{0,1\}^n$ with
  $$ |\cos(z',\xi)| \ge \frac1{\sqrt{2(\log(2K^2)+1)}}. $$
To complete the proof we notice that $\|z'\|\ge\|z\|/\sqrt 2$ by
\refe{half-norm}, and that if $\xi$ is supported on the set of those
$i\in[n]$ with $|z_i|\ge M/2$ (as we can safely assume), then
$\<z',\xi\>=\<z,\xi\>$.
\end{proof}

Finally, we extend Lemma \refl{real-height} onto vectors with \emph{complex}
coordinates.
\begin{lemma}\label{l:ell-complex}
Let $n\ge 1$ be an integer and $K\ge 1$ a real number. If $z\in\C^n$ is a
vector of height $h(z)\le K$, then there exists $\xi\in\{0,1\}^n$ such that
$|\cos(z,\xi)|\ge 1/(2\sqrt{4\log(2K)+2})$.
\end{lemma}

\begin{proof}
Write $z=x+iy$, where $x,y\in\R^n$ and $i$ is the imaginary unit. Assume for
definiteness that $\|x\|\ge\|y\|$, so that $\|x\|\ge\|z\|/\sqrt 2$ in view of
$\|z\|^2=\|x\|^2+\|y\|^2$. Since
  $$ h(x)=\frac{\sqrt{\|x\|_1\|x\|_\infty}}{\|x\|}
            \le \frac{\sqrt{\|z\|_1\|z\|_\infty}}{\|z\|/\sqrt2}
                                              = \sqrt2 h(z) \le \sqrt 2K, $$
by Lemma \refl{real-height} there exists a non-zero $\xi\in\{0,1\}^n$ with
  $$ |\<x,\xi\>| \ge \frac1{2\sqrt{\log(4K^2)+1}}\,\|x\|\|\xi\|
                  \ge \frac{1}{2\sqrt{4\log(2K)+2}}\,\|z\|\|\xi\|. $$
The assertion now follows in view of $|\<x,\xi\>|\le|\<z,\xi\>|$.
\end{proof}

We are eventually ready to prove Theorems~\reft{Dnorm} and~\reft{Pnorm}.

\begin{proof}[Proof of Theorem~\reft{Dnorm}]
Suppose that $A$ is a complex matrix with $m$ rows and $n$ columns, and set
$K:=h(A)$. Since $h(A^\ast)=h(A)$, by Lemma~\refl{tt}, there exists
$z\in\C^m$ such that $\|A^\ast z\|>\frac12\|A^\ast\|\|z\|$ and
$\ell(z)<9K^2$. Write $z=(z_1\longc z_m)^t$ and choose $j\in[m]$ so that
$|z_j|=\min\{|z_i|\colon i\in[1,n],z_i\ne 0\}$. In view of
  $$ h^2(z) = \frac{\|z\|_1\|z\|_\infty}{\|z\|^2}
                = \frac{\|z\|_1|z_j|}{\|z\|^2}\,\ell(z) \le \ell(z) < 9K^2 $$
we then get $h(z)<3K$, whence
\begin{align*}
  h(A^\ast z)
        &= \frac{\sqrt{\|A^\ast z\|_1\|A^\ast z\|_\infty}}{\|A^\ast z\|} \\
        &< \frac{\sqrt{\|A^\ast\|_1\|z\|_1\cdot\|A^\ast\|_\infty\|z\|_\infty}}%
                                                   {\|A^\ast\|\|z\|/2} \\
        &= 2h(A^\ast) h(z) \\
        &< 6K^2,
\end{align*}
and by Lemma~\refl{ell-complex}, there exists $0\ne\xi\in\{0,1\}^n$ with
  $$ |\<A^\ast z,\xi\>| > \frac1{2\sqrt{4\log(12K^2)+2}}
                                                   \,\|A^\ast z\|\|\xi\|. $$
As a result,
  $$ |\<z,A\xi\>| = |\<A^\ast z,\xi\>|
               > \frac1{4\sqrt{4\log(12K^2)+2}}\,\|A\|\|z\|\|\xi\|, $$
implying
\begin{align}
  \|A\xi\| &> \frac1{4\sqrt{4\log(12K^2)+2}}\,\|A\|\|\xi\| \label{e:pt2} \\
           &> \frac1{8\sqrt2\sqrt{\log K+2}}\,\|A\|\|\xi\|. \notag
\end{align}
\end{proof}

\begin{proof}[Proof of Theorem~\reft{Pnorm}]
Observing that the assumptions of Theorems~\reft{Dnorm} and~\reft{Pnorm} are
identical, we re-use the proof of the former theorem, including the notation
$K=h(A)$ and the conclusion that there exists a vector $\xi\in\{0,1\}^n$
satisfying \refe{pt2}. For brevity, denote the denominator of the fraction in
the right-hand side of \refe{pt2} by $f(K)$. Similarly to the computation in
the proof of Theorem~\reft{Dnorm}, and taking into account that $h(\xi)=1$
(as $\xi$ is a binary vector), we obtain
\begin{align*}
  h(A\xi) &= \frac{\sqrt{\|A\xi\|_1\|A\xi\|_\infty}}{\|A\xi\|} \\
          &< \frac{\sqrt{\|A\|_1\|\xi\|_1\cdot\|A\|_\infty\|\xi\|_\infty}}
                            {\|A\|\|\xi\|/f(K)} \\
          &=  Kf(K).
\end{align*}
Applying Lemma~\refl{ell-complex} to the vector $A\xi$, we now find a binary
vector $\eta\in\{0,1\}^m$ with
\begin{align*}
  |\<\eta,A\xi\>|
     &> \frac1{2\sqrt{4\log(2Kf(K))+2}}\,\|A\xi\|\|\eta\| \\
     &> \frac1{2f(K)\sqrt{4\log(2Kf(K))+2}}\,\|A\|\|\xi\|\|\eta\|.
\end{align*}
Finally, it is not difficult to verify that for any $K\ge 1$, the denominator
in the right-hand side is smaller than $32\sqrt2(\log K+4)$, and result
follows.
\end{proof}

\section{Proof of Theorem~\reft{Kn}}\label{s:Kn}

Since
\begin{equation*}
  A_m := \begin{pmatrix} 1 & 1 \\ 1 & 0 \end{pmatrix}^{\otimes m},
\end{equation*}
and since the eigenvalues of the matrix $A_1$ are $\phi:=(1+\sqrt 5)/2$ and
$1-\phi=(1-\sqrt5)/2$, we have $\|A_m\|=\phi^m$.

We split Theorem~\reft{Kn} into two theorems stated in the language and
notation of Corollaries~\refc{Dnorm} and \refc{Pnorm}. These two theorems
will then be given separate proofs.

\renewcommand{\thetheorem}{\reft{Kn}$'$}
\begin{theorem}\label{t:Kn1}
For any integer $m\ge 1$ and subset $X\seq\{0,1\}^m$, writing $N_X(v)$ for
the set of neighbors of a vertex $v\in\{0,1\}^m$ in $X$ (in the graph
$\Gam_m$), we have
  $$ \sum_{v\in\{0,1\}^m} |N_X(v)|^2 \ll \frac{\phi^{2m}}{\sqrt m}\,|X|, $$
with an absolute implicit constant.
\end{theorem}

\renewcommand{\thetheorem}{\reft{Kn}$''$}
\begin{theorem}\label{t:Kn2}
For any integer $m\ge 1$ and subsets $X,Y\seq\{0,1\}^m$, writing $e(X,Y)$ for
the number of edges in $\Gam_m$ joining a vertex from $X$ with a vertex from
$Y$, we have
  $$ e(X,Y) \ll \frac{\phi^{m}}{\sqrt m}\,\sqrt{|X||Y|}, $$
with an absolute implicit constant.
\end{theorem}

We now prepare the technical ground for the proofs of both theorems.

Recall, that the entropy function is defined by
  $$ H(x) := -x\ln x-(1-x)\ln(1-x),\ 0<x<1, $$
extended by continuity onto the endpoints: $H(0)=H(1)=0$.

Let
  $$ \Ome:=\{(x,y)\in\R^2 \colon x\ge 0,\ y\ge 0,\ x+y\le 1 \}, $$
and consider the function
  $$ f(x,y) := (1-x)H\(\frac y{1-x}\)+(1-y)H\(\frac x{1-y}\),
                                                      \quad (x,y)\in\Ome $$
(again, extended by continuity to vanish at the vertex points
$(0,0),\,(0,1)$, and $(1,0)$). Investigating the partial derivatives
  $$ \frac{\partial f}{\partial x} = \ln \frac{(1-x-y)^2}{x(1-x)} $$
and
\begin{equation}\label{e:secder}
  \frac{\partial^2 f}{\partial x^2} = - \frac{1-x-y+2xy}{x(1-x)(1-x-y)} < 0,
\end{equation}
with similar expressions for the derivatives with respect to $y$, we conclude
that $f$ is concave on $\Ome$, and that it is a unimodal function of $x$ for
any fixed $y\in[0,1]$, and a unimodal function of $y$ for any fixed
$x\in[0,1]$. Consequently, the maximum of $f$ on $\Ome$ is attained in the
unique point $(x_0,y_0)\in\Ome$ where both partial derivatives $\partial
f/\partial x$ and $\partial f/\partial y$ vanish; that is,
  $$ \frac{(1-x-y)^2}{x(1-x)} = \frac{(1-x-y)^2}{y(1-y)} = 1. $$
The solution of this system is easily found to be
 $x_0=y_0=(5-\sqrt 5)/10\approx 0.276$, and a simple computation confirms that
the corresponding maximum value is
  $$ f(x_0,y_0) = 2\ln\phi. $$

We will also need well-known estimates for the binomial coefficients which
can be easily derived, for instance, from \cite[Ch. 10,~\S 11, Lemmas~7
and~8]{b:mcwsl}:
\begin{equation}\label{e:binomial1}
  \frac1{\sqrt{2m}} \: e^{mH(k/m)} \le \binom mk
               \le \sum_{i=0}^k \binom mi \le e^{mH(k/m)},\quad 0\le k\le m/2,
\end{equation}
and
\begin{equation}\label{e:binomial2}
  \sum_{i=0}^k \binom mi \ll_\eps \frac1{\sqrt m}\:e^{mH(k/m)},
                                                 \quad 1\le k\le (1-\eps) m/2,
\end{equation}
for any $\eps>0$ (with the implicit constant depending on $\eps$).

The following lemma is used in the proof of Theorem~\reft{Kn1}.

\begin{lemma}\label{l:Kn1}
For integer $m\ge 0$ and $j\in[0,m]$, let
  $$ \tau_m(j) := \sum_{i=0}^{m-j} \binom{m-i}{j} \binom{m-j}{i}. $$
Then
  $$ \max \{ \tau_m(j)\colon j\in[0,m] \} \ll \frac{\phi^{2m}}{\sqrt m}, $$
with an absolute implicit constant.
\end{lemma}

\begin{proof}
We use the notation introduced at the beginning of this section; thus, for
instance, in view of \refe{binomial1},
\begin{equation}\label{e:binprod}
  \binom{m-i}{j} \binom{m-j}{i}
                  \le e^{(m-i)H(j/(m-i))+(m-j)H(i/(m-j))} = e^{mf(i/m,j/m)}.
\end{equation}

Let $I:=(0.2,0.3)$. Since $x_0=y_0\in I$, we have
 $\max_{\Ome\stm(I\times I)}f<2\ln\phi$; therefore, by \refe{binprod}, we can
fix $B<\phi^2$ so that
\begin{equation}\label{e:jnotin}
  \tau_m(j)=O(mB^m),\quad j/m\notin I,
\end{equation}
and also
\begin{equation}\label{e:taumj1}
  \tau_m(j) = \sum_{\sub{0\le i\le m-j \\ i/m\in I}}
                    \binom{m-i}{j} \binom{m-j}{i} + O(mB^m),\quad j/m\in I.
\end{equation}
For every pair $(i,j)$ with $(i/m,j/m)\in I\times I$, we have
  $$ \frac14 = \frac{0.2m}{m-0.2m} < \frac{i}{m-j}
                                         < \frac{0.3m}{m-0.3m} = \frac37, $$
and by symmetry, the resulting estimate holds true also for the ratio
$j/(m-i)$; consequently, in view of \refe{taumj1} and \refe{binomial2}, if
$j/m\in I$, then
\begin{align}\label{e:approxf}
  \tau_m(j)
    &\le \sum_{\sub{0\le i\le m-j \\ i/m\in I}}
            \frac1{\sqrt m}\, e^{(m-i)H(j/(m-i))}
              \cdot \frac1{\sqrt m}\, e^{(m-j)H(i/(m-j))} +O(mB^m) \notag \\
    &\le \frac1m \, \sum_{i=0}^{m-j} e^{mf(i/m,j/m)} + O(mB^m).
\end{align}
Since $f(x,j/m)$ is a concave function of $x$ for any fixed $j\in[0,m]$, on
each interval of the form $[i/m,(i+1)/m]$ it attains its minimum value at one
of the endpoints of the interval, and so does the function $e^{mf(x,j/m)}$.
Hence,
\begin{align*}
  \int_{i/m}^{(i+1)/m} e^{mf(x,j/m)}\, dx
        &\ge \frac1m\, \min \{ e^{mf(x,j/m)}\colon i/m\le x\le (i+1)/m \} \\
        &=   \frac1m\, \min \{ e^{mf(i/m,j/m)}, e^{mf((i+1)/m,j/m)} \};
                                                      \quad 0\le i\le m-j-1.
\end{align*}
Similarly, unimodality of $f(x,j/m)$ on the interval $x\in[0,1-j/m]$ implies
that of $e^{mf(x,j/m)}$; as a result, adding up for all $i\in[0,m-1-j]$ the
estimate just obtained, we get
\begin{align}\label{e:sum2int}
  \frac1m\,\sum_{i=0}^{m-j} e^{mf(i/m,j/m)}
    &\le \int_0^{1-j/m} e^{mf(x,j/m)}\,dx
         + \frac1m\,\max \{ e^{mf(i/m,j/m)} \colon 0\le i\le m-j \} \notag \\
    &\le \int_0^{1-j/m} e^{mf(x,j/m)}\,dx + \frac{\phi^{2m}}m.
\end{align}

We now use the second-order polynomial approximation to show that
\begin{equation}\label{e:saddle}
  f(x,y) \le 2\ln\phi-\frac23(x-x_0)^2,\quad (x,y)\in\Ome;
\end{equation}
substituting this estimate into \refe{sum2int} will eventually allow us to
compete the proof of the lemma.

Let $z_0:=x_0/(1-x_0)$. A simple computation confirms that
\begin{gather*}
  z_0     = 2-\phi\approx 0.382, \\
  H'(z_0) = \ln(z_0^{-1}-1) = \ln\phi,
  \intertext{and}
  H''(z) = -\frac1{z(1-z)}\le -4,\quad z\in(0,1);
\end{gather*}
consequently, by Taylor's formula,
  $$ H(z) \le H(z_0) + (z-z_0)\ln\phi - 2(z-z_0)^2,\quad z\in(0,1). $$
Applying this estimate with $z=y/(1-x)$ and multiplying the result by $1-x$,
in view of $1/(1-x)\ge 1$ we get
  $$ (1-x)H\(\frac y{1-x}\) \le (1-x)H(z_0) + (y-(1-x)z_0)\ln\phi
                                           - 2(y-(1-x)z_0)^2. $$
Interchanging $x$ and $y$ and adding the resulting estimate to the one just
obtained yields
\begin{equation}\label{e:f-decomp}
  f(x,y) \le L(x,y) - 2Q(x,y),\quad (x,y)\in\Ome,
\end{equation}
where
  $$ L(x,y) = (2-x-y)H(z_0)+(x+y-(2-x-y)z_0)\ln\phi $$
and
  $$ Q(x,y) = (x-(1-y)z_0)^2 + (y-(1-x)z_0)^2. $$
One easily verifies that $H(z_0)=(z_0+1)\ln\phi$ and, as a result, the linear
part is actually constant:
\begin{equation}\label{e:lin-part}
  L(x,y)=2\ln\phi.
\end{equation}
To estimate the quadratic part we set $\xi:=x-x_0$ and $\eta:=y-y_0$; with
this notation, and taking into account that $x_0=(1-y_0)z_0$ and
$y_0=(1-x_0)z_0$, we have
\begin{align}\label{e:quadr-part}
  Q(x,y) &=   (\xi+z_0\eta)^2 + (\eta+z_0\xi)^2 \notag \\
         &=   (z_0^2+1)(\xi^2+\eta^2) + 4z_0\xi\eta \notag \\
         &\ge (z_0-1)^2(\xi^2+\eta^2) \notag \\
         &=   z_0(\xi^2+\eta^2) \notag \\
         &\ge \frac13 (x-x_0)^2.
\end{align}

From \refe{f-decomp}, \refe{lin-part}, and \refe{quadr-part} we get the
desired estimate \refe{saddle}. Substituting it into \refe{sum2int} and
recalling \refe{approxf}, we obtain
\begin{align*}
  \tau_m(j)
      &\le \phi^{2m} \int_0^{1-j/m} e^{-(2/3)m(x-x_0)^2}\,dx
                                                     + O(\phi^{2m}/m) \\
      &<   \phi^{2m} \int_{-\infty}^\infty e^{-(2/3)m(x-x_0)^2}\,dx
                                                     + O(\phi^{2m}/m) \\
      &=   O(\phi^{2m}/\sqrt m), \quad j/m\in I;
\end{align*}
along with \refe{jnotin}, this proves the lemma.
\end{proof}

We are now ready for the proofs of Theorems~\reft{Kn1} and~\reft{Kn2}.
\begin{proof}[Proof of Theorem~\reft{Kn1}]
Writing for brevity
  $$ \sig(X) := \sum_{v\in\{0,1\}^m} |N_X(v)|^2, $$ we want to prove that
\begin{equation}\label{e:sigX-toprove}
  \sig(X) \ll \frac{\phi^{2m}}{\sqrt m}\, |X|
\end{equation}
for every subset $X\seq\{0,1\}^m$.

For a vector $v\in\R^m$, let $|v|$ denote the number of non-zero coordinates
of $v$; thus, for instance, if $v\in\{0,1\}^m$, then $|v|=\|v\|^2$. Since,
for any $v\in\{0,1\}^m$, the total number of neighbors of $v$ in $\Gam_m$ is
$2^{m-|v|}$, we have
  $$ \sig(X) \le \sum_{v\in\{0,1\}^m} |N_X(v)|^2
        \le \sum_{v\in\{0,1\}^m} 4^{m-|v|} = 5^m < \phi^{2m}\cdot 1.91^m, $$
establishing \refe{sigX-toprove} in the case where $|X|\ge 1.92^m$. On the
other hand, $|N_X(v)|\le|X|$ implies
  $$ \frac{\sig(X)}{|X|} \le \sum_{v\in\{0,1\}^m} |N_X(v)| \le 2^m|X|, $$
and if $|X|\le 1.3^m$, then the right-hand side does not exceed $2.6^m$,
whereas $\phi^2>2.61$. With these observations in mind, for the rest of the
proof we assume that
\begin{equation}\label{e:XLowerUpper}
  1.3^m < |X| < 1.92^m.
\end{equation}

For $r\in[0,m]$, write $B_r:=\{v\in\{0,1\}^m\colon |v|\le r \}$; thus,
  $$ |B_r| = \sum_{i=0}^r \binom mi. $$
Let $q\in[1,m-1]$ be defined by
  $$ |B_{q-1}| < |X| \le |B_q|. $$
In view of \refe{XLowerUpper} and \refe{binomial1}, this implies
\begin{equation}\label{e:qm}
  cm < q < Cm
\end{equation}
with some absolute constants $0<c<C<1/2$; consequently,
  $$ \frac{|B_q|}{|B_{q-1}|} \le 1 + \binom mq\Big/\binom m{q-1}
                                            = 1 + \frac{m-q+1}{q} = O(1). $$
It follows that for any set $Y\seq\{0,1\}^m$ with $X\seq Y$ and $|Y|=|B_q|$
we have
  $$ \sig(X)/|X| \le (\sig(Y)/|Y|) \cdot (|Y|/|X|) \ll \sig(Y)/|Y|, $$
showing that it suffices to prove \refe{sigX-toprove} under the assumption
$|X|=|B_q|$.

Using partial summation, we get
\begin{align*}
  \sig(X) &= \sum_{x,y\in X} |\{v\in\{0,1\}^m\colon \<x+y,v\>=0\}| \\
          &= \sum_{x,y\in X} 2^{m-|x+y|} \\
          &= \sum_{k=0}^m 2^{m-k} |\{(x,y)\in X\times X\colon |x+y|=k \}| \\
          &= \sum_{k=0}^{m-1}
                     2^{m-1-k} |\{(x,y)\in X\times X\colon |x+y|\le k \}| \\
          &{\hskip 1.75in}  + |\{(x,y)\in X\times X\colon |x+y|\le m \}|,
\end{align*}
and we now apply a result of Bollob\'as and Leader \cite[Corollary~4]{b:bole}
which says (in a dual form, and in the language of set families) that if
$q\in[0,m]$, and $X$ is a set of $m$-dimensional binary vectors with
$|X|=|B_q|$, then for any integer $k\in[0,m]$, the number of pairs $(x,y)\in
X\times X$ with $|x+y|\le k$ is maximized when $X=B_q$. As a result, we can
replace our present assumption $|X|=|B_q|$ with the stronger assumption
$X=B_q$.

For $r\in[0,m]$, write $S_r:=\{v\in\{0,1\}^m\colon |v|=r\}$; thus,
$B_q=S_0\cup\ldots\cup S_q$, and, in view of \refe{qm},
  $$ \frac{|S_{r-1}|}{|S_r|} = \binom{m}{r-1}\big/\binom mr
        = \frac{r}{m-r+1} < \frac{C}{1-C} < 1,\quad 1\le r\le q, $$
implying
\begin{equation}\label{e:Si}
  \sum_{r=0}^q (q+1-r)^2 |S_r| \le
          \sum_{r=0}^q \(\frac{C}{1-C}\)^{q-r} (q+1-r)^2 |S_q|
                                                       \ll |S_q| \le |B_q|.
\end{equation}
We now claim that to prove \refe{sigX-toprove} with $X=B_q$, it suffices to
prove it in the case where $X=S_r$, for all $r\in[0,m]$. To see this, we
notice that if \refe{sigX-toprove} is established in this special case, then,
by the Cauchy-Schwartz inequality and \refe{Si},
\begin{align*}
  \sig(B_q)
    &= \sum_{v\in\{0,1\}^m} \Bigg( \sum_{r=0}^q (q+1-r)\sqrt{|S_r|}
                    \cdot \frac{|N_{S_r}(v)|}{(q+1-r)\sqrt{|S_r|}} \Bigg)^2 \\
    &\le \sum_{v\in\{0,1\}^m} \Bigg( \sum_{r=0}^q (q+1-r)^2|S_r| \Bigg)
               \sum_{r=0}^q \frac1{(q+1-r)^2}\,\frac{|N_{S_r}(v)|^2}{|S_r|} \\
    &\ll |B_q| \sum_{r=0}^q \frac1{(q+1-r)^2}\,\frac1{|S_r|}
                                        \sum_{v\in\{0,1\}^m} |N_{S_r}(v)|^2 \\
    &\ll \frac{\phi^{2m}}{\sqrt m}\, |B_q| \sum_{r=0}^q \frac1{(q+1-r)^2} \\
    &\ll \frac{\phi^{2m}}{\sqrt m}\, |B_q|.
\end{align*}
We thus can assume that $X=S_r$ for some $r\in[0,m]$. Therefore,
  $$ |N_X(v)| = \begin{cases}
                \binom{m-|v|}r\ &\text{if}\ |v|\le m-r, \\
                0               &\text{if}\ |v|>m-r.
              \end{cases} $$
Consequently,
  $$ \sig(X)/|X|
       = \sum_{i=0}^{m-r} \binom mi \binom{m-i}r^2 \Big/ \binom mr
       = \sum_{i=0}^{m-r} \binom{m-i}r \binom{m-r}i, $$
and the result now follows from Lemma~\refl{Kn1}.
\end{proof}

\begin{proof}[Proof of Theorem~\reft{Kn2}]
Suppose that $m\ge 1$ and $\est\ne X,Y\seq\{0,1\}^m$; we want to show that
$e(X,Y)\ll(\phi^{m}/\sqrt m)\sqrt{|X||Y|}$.

We start with the observation that if there is a vertex $x\in X$ with
$|N_Y(x)|<e(X,Y)/(2|X|)$, then, letting $X':=X\stm\{x\}$, we have $X'\ne\est$
and
  $$ \frac{e(X',Y)}{\sqrt{|X'||Y|}} \ge \frac{e(X,Y)}{\sqrt{|X||Y|}}: $$
this follows readily from $e(X',Y)=e(X,Y)-|N_Y(x)|$ and $|X'|=|X|-1$. A
similar remark applies to the vertices $y\in Y$ having ``too few'' neighbors
in $X$. Repeating this procedure, we ensure that $|N_Y(x)|\ge e(X,Y)/(2|X|)$
for every vertex $x\in X$, and that $|N_X(y)|\ge e(X,Y)/(2|Y|)$ for every
vertex $y\in Y$.

We keep using the notation $|v|$ for the number of non-zero coordinates of a
vector $v\in\R^m$. Let $m_1:=\max\{|x|\colon x\in X\}$, and choose
arbitrarily a vertex $x\in X$ with $|x|=m_1$. Similarly, let
$m_2:=\max\{|y|\colon y\in Y\}$ and choose $y\in Y$ with $|y|=m_2$. We have
  $$ \frac{e(X,Y)}{2|X|} \le |N_Y(x)| \le \sum_{k=0}^{m_2} \binom{m-m_1}{k} $$
and
  $$ \frac{e(X,Y)}{2|Y|} \le |N_X(y)| \le \sum_{k=0}^{m_1} \binom{m-m_2}{k}. $$
To complete the proof, we now show that
  $$ P := \sum_{k=0}^{m_2} \binom{m-m_1}{k}
             \cdot \sum_{k=0}^{m_1} \binom{m-m_2}{k} \ll \frac{\phi^{2m}}{m} $$
uniformly in $m_1,m_2\in[0,m]$.

Assume for definiteness that $m_1\le m_2$. If $m_2>(m-m_1)/2$, then replacing
$m_2$ with $\lfloor(m-m_1)/2\rfloor$ enlarges the second factor in the
definition of $P$, whereas the first factor can get at most twice smaller. As
a result, we can assume that
\begin{equation}\label{e:n2small}
  m_2\le(m-m_1)/2,
\end{equation}
and (in view of $m_1\le m_2$) also that
\begin{equation}\label{e:n1small}
  m_1\le(m-m_2)/2;
\end{equation}
consequently,
  $$ 0\le m_1,m_2\le m/2. $$

Write $\mu_i:=m_i/m\ (i\in\{1,2\})$. Taking into account \refe{n2small} and
\refe{n1small}, by \refe{binomial1} we get
  $$ P \le e^{f(\mu_1,\mu_2)m}, $$
and if both $\mu_2/(1-\mu_1)$ and $\mu_1/(1-\mu_2)$ are bounded away from
$1/2$, then indeed
\begin{equation}\label{e:strong-bound}
  P \ll \frac 1m\,e^{f(\mu_1,\mu_2)m}
\end{equation}
by \refe{binomial2}.

Let $\Ome_0:=[0,0.3)^2$, and write $M:=\max_{\Ome\stm\Ome_0} f$. Since the
maximum of $f$ on $\Ome$ is attained at the unique point $(x_0,y_0)\in\Ome_0$
(as explained at the beginning of this section), we have
$M<f(x_0,y_0)=2\ln\phi$; hence,
 $P\le e^{mM}=o(\phi^{2m}/m)$ for $(\mu_1,\mu_2)\notin\Ome_0$. On the other
hand, if $(\mu_1,\mu_2)\in\Ome_0$, then
  $$ \frac{\mu_1}{1-\mu_2} \le \frac37
                 \quad \text{and}\quad \frac{\mu_1}{1-\mu_2} \le \frac37, $$
which in view of \refe{strong-bound} gives
  $$ P \ll \frac1m\,e^{f(x_0,y_0)m} = \frac{\phi^{2m}}m. $$
This completes the proof of Theorem~\reft{Kn2}.
\end{proof}

\section*{Acknowledgements}
The author is grateful to Noga Alon for the idea behind the proof of
Theorem~\reft{Kn2}, and to Terry Tao for the proof of Lemma~\refl{tt}.

\ifthenelse{\equal{\showall}{yes}}{}{\bigskip\enddocument}

\newpage
\appendix
\section{Lemma~\refl{tt} for Hermitian matrices}

A version of Lemma~\refl{tt} for Hermitian matrices is as follows.
\begin{lemma}\label{l:tt-square}
Let $n\ge 1$ be an integer and $K\ge 1$ a real number. For any non-zero
Hermitian matrix $A$ with $\|A\|_\infty\le K\|A\|$, there exists a vector
$z\in\C^n$ such that $\|Az\|>\frac14\|A\|\|z\|$ and $\ell(z)<8K+1$.
\end{lemma}

\begin{proof}
Fix a unit-length vector $x=(x_1\longc x_n)\in\C^n$ with $|\<Ax,x\>|=\|A\|$
and let $M:=8K+1$. Consider the decomposition
  $$ x = \sum_{k=-\infty}^\infty x^{(k)}, $$
where for every integer $k$, the vector $x^{(k)}=(x^{(k)}_1\longc x^{(k)}_n)$
is defined by
  $$ x^{(k)}_i := \begin{cases}
       x_i &\text{ if}\ M^k\le |x_i|<M^{k+1}, \\
       0   &\text{ otherwise}.
                  \end{cases} $$
Notice, that $\ell(x^{(k)})<M=8K+1$ whenever $x^{(k)}\ne 0$. Furthermore, we
have
\begin{equation}\label{e:Axx-large-app}
  \|A\| = |\<Ax,x\>| \le \sum_{k,l=-\infty}^\infty |\<Ax^{(k)},x^{(l)}\>|
\end{equation}
and, since the vectors $x^{(k)}$ are pairwise orthogonal,
  $$ \sum_{k=-\infty}^\infty \|x^{(k)}\|^2 = \|x\|^2 = 1. $$
Recalling that $\|A\|_\infty\le K\|A\|$ implies
 $|\<Au,v\>|\le K\|A\|\|u\|_\infty\|v\|_1$ for all $u,v\in\C^n$, the
contribution to the right-hand side of \refe{Axx-large-app} of the summands
with $l\ge k+2$ can be estimated as follows:
\begin{align*}
  \sum_{k,l\colon l\ge k+2} |\<Ax^{(k)},x^{(l)}\>|
    &\le K\|A\| \sum_{k,l\colon l\ge k+2} M^{k+1}
                           \sum_{i\in[n]\colon M^l\le |x_i|<M^{l+1}} |x_i| \\
    &= K\|A\| \sum_{l=-\infty}^\infty
                  \sum_{i\in[n]\colon M^l\le |x_i|<M^{l+1}} |x_i|
                                            \sum_{k=-\infty}^{l-2} M^{k+1} \\
    &\le \frac{K}{M-1}\,\|A\|\, \sum_{l=-\infty}^\infty
                         \sum_{i\in[n]\colon M^l\le |x_i|<M^{l+1}} |x_i|^2 \\
    &=   \frac18\,\|A\|\|x\|^2 \\
    &=   \frac18\,\|A\|.
\end{align*}
Since $A$ is Hermitian, we conclude that
\begin{equation}\label{e:off-diag-app}
  \sum_{k,l\colon|k-l|\ge 2} |\<Ax^{(k)},x^{(l)}\>| \le \frac14\,\|A\|.
\end{equation}

Assuming that the assertion of the theorem fails to hold, we have
$\|Ax^{(k)}\|\le\frac14\,\|A\|\|x^{(k)}\|$ for every integer $k$. Hence,
under this assumption, for any fixed integer $d$,
\begin{align*}
  \sum_{k,l\colon k-l=d} |\<Ax^{(k)},x^{(l)}\>|
    &\le \frac14\,\|A\| \sum_{l=-\infty}^\infty \|x^{(l)}\|\|x^{(l+d)}\| \\
    &\le   \frac14\,\|A\|\sum_l \|x^{(l)}\|^2 \\
    &= \frac14\,\|A\|,
\end{align*}
the second inequality being strict unless $d=0$. It follows that
  $$ \sum_{k,l\colon|k-l|\le 1} |\<Ax^{(k)},x^{(l)}\>| < \frac34\,\|A\|; $$
along with \refe{off-diag-app} this yields
  $$ \sum_{k,l=-\infty}^\infty |\<Ax^{(k)},x^{(l)}\>| < \|A\|, $$
contradicting \refe{Axx-large-app}.
\end{proof}


\vfill
\begin{thebibliography}{McWS77}
\bibitem[AS08]{b:as}
  {\sc N.~Alon} and {\sc J.H. Spencer},
  \emph{The probabilistic method. Second edition.}
  Wiley-Interscience Series in Discrete Mathematics and Optimization,
  Wiley-Interscience [John Wiley \& Sons], New York, 2000.
\bibitem[BL06]{b:bili}
  {\sc Y.~Bilu} and {\sc N.~Linial},
  Lifts, discrepancy, and near-optimal spectral gap,
  \emph{Combinatorica} {\bf 26} (5) (2006), 495--519.
\bibitem[BL03]{b:bole}
  {\sc B.~Bollob\'as} and {\sc I.~Leader},
  Set systems with few disjoint pairs,
  \emph{Combinatorica} {\bf 23} (4) (2003), 559--570.
\bibitem[BN04]{b:bn}
  {\sc B.~Bollob\'as} and {\sc V.~Nikiforov},
  Hermitian matrices and graphs: singular values and discrepancy,
  \emph{Discrete Math.} {\bf 285} (1–3) (2004), 17--32.
\bibitem[McWS77]{b:mcwsl}
  {\sc F.J. MacWilliams and N.J.A. Sloane},
  \emph{The Theory of Error-Correcting Codes},
  North-Holland (1977).
\bibitem[M60]{b:m}
  {\sc L.~Mirsky},
  Symmetric gauge functions and unitarily invariant norms,
  \emph{Quart. J.~Math. Oxford Ser.} (2) {\bf 11} (1960), 50--59.
\bibitem[S93]{b:s}
  {\sc Stewart, G.W.},
  On the early history of the singular value decomposition,
  \emph{SIAM Rev.} {\bf 35} (4) (1993), 551--566.
\end{thebibliography}
\end{document}